\documentclass[12pt]{article}
\usepackage{latexsym}
\usepackage{multirow}
\usepackage{graphicx}
\usepackage{float}
\usepackage{authblk}
\usepackage{multirow}
\usepackage{kantlipsum}
\usepackage{setspace}
\usepackage{xcolor, color}
\usepackage{bbm}  
\def\0{\emptyset}

\usepackage{indentfirst}
\usepackage{amsmath}

\def\0{\emptyset}

\usepackage{mathrsfs}
\usepackage{amssymb}
\usepackage{listings}
\usepackage{amsmath}
\usepackage{amsfonts}
\usepackage{graphicx}
\usepackage[small]{caption}
\usepackage{subfigure}
\usepackage{cite}
\usepackage{mathtools}
\usepackage{multirow}
\usepackage{tikz}
\newtheorem{theorem}{Theorem}[section]

\newtheorem{definition}[theorem]{Definition}

\newtheorem{lemma}[theorem]{Lemma}
\newtheorem{prop}[theorem]{Proposition}

\newenvironment{proof}{\noindent {\bf Proof.}}{\rule{2mm}{2mm}\par\medskip}

\newcommand\blue[1] {{\bf \color{blue} #1}}

\newcommand\equ[2]
{
	\begin{equation} \label{#1}
		#2
	\end{equation} 
}

\newcommand\eqn[2]
{
	\begin{eqnarray} \label{#1}
		#2
	\end{eqnarray} 
}

\newcounter{countclaim}
\def\inclaim{\addtocounter{countclaim}{1}
	{\vspace{0.2 cm}
		\noindent {\bf Claim \thecountclaim}: }}

\newcounter{countcase}
\def\incase{\addtocounter{countcase}{1}
	{\vspace{0.2 cm}
		\noindent {\bf Case \thecountcase}: }}

\newcounter{countsubcase}

\begin{document}
\title{A neighborhood union condition for the existence of a spanning tree without samll degree vertices\thanks{This research was partially supported by NSFC (No. 12371345 and 12371340) and CSC scholarship program (No. 202408420212).}}

\author[1,2]{\small Yibo Li\thanks{Email: liyibo@stu.hubu.edu.cn}}

\author[3]{\small Fengming Dong\thanks{Corresponding author. Email: fengming.dong@nie.edu.sg and donggraph@163.com}}

\author[1,2]{\small Huiqing Liu\thanks{ Email: hqliu@hubu.edu.cn}}

\affil[1]{\footnotesize Hubei Key Laboratory of Applied Mathematics, Faculty of Mathematics and Statistics, 
	
	Hubei University, Wuhan 430062, China}

\affil[2]{\footnotesize Key Laboratory of Intelligent Sensing System and Security (Hubei University), Ministry of Education}

\affil[3]{\footnotesize National Institute of Education, Nanyang Technological University, Singapore}

\date{}
\maketitle

\begin{abstract}
		For an integer $k\ge2$, a $[2,k]$-ST of a connected graph $G$ is a spanning tree of $G$ in which there are no vertices of degree between $2$ and $k$. A $[2,k]$-ST is a natural extension of a homeomorphically irreducible spanning tree (HIST), which is a spanning tree without vertices of degree $2$. In this paper, we give a neighborhood union condition for the existence of a $[2,k]$-ST in $G$. 
		We generalize a known degree sum condition that guarantees the existence of a $[2,k]$-ST in $G$.
\end{abstract}



\section{Introduction}
All graphs considered in this paper are simple and finite. For any graph $G$, let $V(G)$ and $E(G)$ denote its vertex set and edge set, respectively. For $v\in V(G)$, $N_G(v)$ is the {\em set of neighbors} of $v$, and $d_G(v)=|N_G(v)|$ is the {\em degree} of $v$ in $G$. Denote $N_G[v]=N_G(v)\cup\{v\}$. We may simplify write $d(v)$
and  $N[v]$ 
for $d_G(v)$ and $N_G[v]$,
respectively, 
if there is no risk of confusion.
Let $|G|$ and $\delta(G)$ be the {\em number of vertices} and the {\em minimum degree} of $G$, respectively. For an integer $i\ge0$, let $V_i(G)=\{u\in V(G):d_G(u)=i\}$. If $G$ is not a complete graph, then we define
\begin{eqnarray*}
	\sigma(G)&:=&\mbox{min}\{d_G(u)+d_G(v):u,v\in V(G),
	\ u\neq v, uv\notin E(G)\},\\
	NC(G)&:=&\mbox{min}\{|N_G(u)\cup N_G(v)|:
	u,v\in V(G),\ u\neq v,
	uv\notin E(G)\}.
\end{eqnarray*}

For any non-empty subset $S$ of $V(G)$, let $G[S]$ denote the subgraph of $G$ induced by $S$, let $G-S$ denote $G[V(G)\setminus S]$
when $S\ne V(G)$, and write $N_S(v)$ and $d_S(v)$ for $N_G(v)\cap S$ and $|N_G(v)\cap S|$, respectively, for each  $v\in V(G)$. If $S=\{v\}$, then we simplify $G-\{v\}$ to $G-v$. For any proper subgraph $H$ of $G$ and $S\subseteq V(G)\setminus V(H)$, 
let $H+S$ 
denote $G[V(H)\cup S]$
and simplify $H+\{v\}$ to
$H+v$  if $S=\{v\}$.

  Given two disjoint vertex sets $X,Y\subseteq V(G)$, 
  let $E_G(X,Y)$
  (or simply $E(X,Y)$) 
  denote the set of 
  edges $xy$ with $x\in X$ and $y\in Y$. 
  When $X=\{x\}$, we write $E(x,Y)$ for $E(X,Y)$. 
  For a subgraph $H$ of $G$, we consider it as both a subgraph and
a vertex set of $G$. Denote by $K_n$ the complete graph of order $n$.

The study of spanning trees under specific degree constraints is a central topic in graph theory, particularly in the study of Hamiltonian paths and their generalizations. For a graph $G$, a spanning tree of $G$ without vertices of degree $2$ is called a {\em homeomorphically irreducible spanning tree} (HIST), that is, a spanning tree $T$ of $G$ is a HIST if and only if $V_2(T) = \emptyset$. A HIST can be viewed as a natural counterpart to a Hamiltonian path, which is a spanning tree where every vertex, except for the endvertices, has degree exactly 2. As Hamiltonian path research progresses, HIST research has been attracting attention (see \cite{ABHT90, ChRS12, Chsh13, FuTs20, ItTs22, LiDo24,FuTs13} for example). 

Various sufficient conditions for the existence of a HIST have been established. For instance, Albertson, Berman, Hutchinson and Thomassen \cite{ABHT90} gave a condition on $\delta(G)$ for the existence of a HIST.

\begin{theorem} [\cite{ABHT90}]\label{Thm1.1}
	Let $G$ be a connected graph of order $n$. If $\delta(G)\geq4\sqrt{2n}$, then $G$ has a HIST.
\end{theorem}

Later, Ito and Tsuchiya \cite{ItTs22} found a sufficient condition based on the minimum degree sum $\sigma(G)$.

\begin{theorem}[\cite{ItTs22}]	\label{Thm1.2}
	Let $G$ be a graph of order $n\geq8$.
	If 
	$\sigma(G)\geq n-1$, 
	then $G$ has a HIST.
\end{theorem}





In \cite{LiDo24},
the authors of this article
	together with another author
found that  $NC(G)\ge\frac{n-1}{2}$ 
is a weaker sufficient condition 
for the existence of a HIST.

\begin{theorem}[\cite{LiDo24}] \label{Thm1.3} 
	Let $G$ be a connected graph of order $n\geq270$. If 
	\begin{equation*} \label{Thm1.1.3-e1} 
		NC(G)\geq\frac{n-1}{2},
	\end{equation*} 
	then $G$ has a HIST
	if and only if $G$ does not belong to the four exceptional families of graphs.  
\end{theorem} 

For an intrger $k\ge2$, a spanning tree $T$ of $G$ is called a {\em $[2,k]$-ST} of $G$ if $V_i(T)=\emptyset$
for each $i$ with $2\le i\le k$.
Note that a $[2,2]$-ST is exactly a HIST. This concept, introduced by Furuya and Tsuchiya in \cite{FuTs13}, generalizes the definition of a HIST. Similar to the study of HIST, it is natural to consider the existence of a $[2,k]$-ST in terms of  degree conditions. Some related results on the existence of $[2,k]$-STs have been obtained (see~\cite{FuST24, FuTsarxiv}).  

For any integer $k\ge 2$, let $c_k
 =(\sqrt{k}+\sqrt{2})\sqrt{k(k-1)}$. 
 Furuya, Saito, and Tsuchiya~\cite{FuST24} established a minimum degree condition for the existence of a $[2,k]$-ST, which generalized Theorem~\ref{Thm1.1}.

\begin{theorem}
[\cite{FuST24}]
	\label{Thm1.4}
	Let $k\ge2$ be an integer, and let $G$ be a connected graph of order $n$. If
	\begin{equation} \label{Thm1.4-e1} 
		\delta(G)\ge c_k\sqrt{n},
	\end{equation}
	then $G$ has a $[2,k]$-ST.
\end{theorem}


Let $n_0(k)$ be the smallest positive integer such that $n-4c_k\sqrt{n}-2k^2-4k-4\ge0$
holds for every integer $n\ge n_0(k)$. In \cite{FuTsarxiv}, Furuya and Tsuchiya gave a degree sum condition for the existence
of a $[2,k]$-ST.

\begin{theorem}[\cite{FuTsarxiv}]\label{Thm1.5}
	Let $k\ge2$ be an integer, and let $G$ be a connected graph of order $n\ge n_0(k)$. If
	$$
	\sigma(G)\ge n-2,
	$$ 
	then $G$ has a $[2,k]$-ST if and only if $G$ dose not belong to one exceptional family of graphs.
\end{theorem}

\vspace{3mm}
For sufficiently large graphs, Theorem~\ref{Thm1.5} is a generalization of Theorem~\ref{Thm1.2}. In this paper, we consider the existence of a $[2,k]$-ST in a graph $G$ under the condition that $NC(G)\ge\frac{n-2}{2}$. Let $n_1(k)$ be the smallest positive integer such that $n-4c_k\sqrt{n}-12k+14\ge0$ holds for every integer $n\ge n_1(k)$. It is easy to 
verify that 
for $k=2,3,4,5$, we have
$n_1(k)=276, 994, 2306$
and $4356$, 
respectively.
It can also be proved that 
$n_1(k)>16k^3$ for all $k\ge 2$.


In this article, we establish the following conclusion.

\begin{theorem}\label{Thm1.6}
	Let $k\ge2$ be an integer, and let $G$ be a connected graph of order $n\ge n_1(k)$. If $\delta(G)\ge2k$ and		
	\begin{equation} \label{Thm1.6-e1} 
		NC(G)\geq\frac{n-2}{2},
	\end{equation} 
	then $G$ has a $[2,k]$-ST.
\end{theorem} 

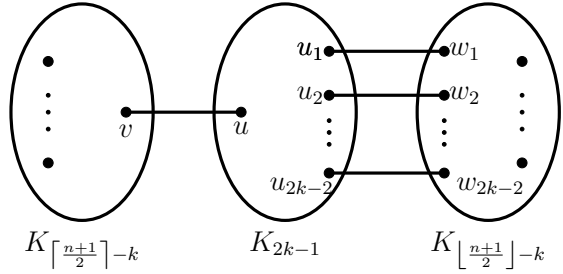
\begin{figure}[h]
\begin{center}
	\begin{tikzpicture}[scale=0.45]
	    \tikzstyle{every node}=[font=\normalsize,scale=0.9]
		\draw[very thick](0,0) ellipse [x radius=2.1cm, y radius=3.2cm];
		\draw[very thick](6,0) ellipse [x radius=2.1cm, y radius=3.2cm];
		\draw[very thick](12,0) ellipse [x radius=2.1cm, y radius=3.2cm];
		\filldraw (7.3,1.8)circle(0.8ex);
		\filldraw (7.3,0.5)circle(0.8ex);
		\filldraw (7.3,-0.25)circle(0.3ex);
		\filldraw (7.3,-0.6)circle(0.3ex);
		\filldraw (7.3,-0.95)circle(0.3ex);
		\filldraw (7.3,-1.8)circle(0.8ex);
		\filldraw (4.7,0)circle(0.8ex);
		\filldraw (10.7,1.8)circle(0.8ex);
		\filldraw (10.7,0.5)circle(0.8ex);
		\filldraw (10.7,-0.25)circle(0.3ex);
		\filldraw (10.7,-0.6)circle(0.3ex);
		\filldraw (10.7,-0.95)circle(0.3ex);
		\filldraw (10.7,-1.8)circle(0.8ex);
		\filldraw (13,1.5)circle(0.8ex);
		\filldraw (13,-1.5)circle(0.8ex);
		\filldraw (13,0.5)circle(0.3ex);
		\filldraw (13,0)circle(0.3ex);
		\filldraw (13,-0.5)circle(0.3ex);
		\draw[very thick] (7.3,1.8)--(10.7,1.8);
		\draw[very thick] (7.3,0.5)--(10.7,0.5);
		\draw[very thick] (7.3,-1.8)--(10.7,-1.8);
		\filldraw (1.3,0)circle(0.8ex);
		\filldraw (-1,1.5)circle(0.8ex);
		\filldraw (-1,-1.5)circle(0.8ex);
		\filldraw (-1,0.5)circle(0.3ex);
		\filldraw (-1,0)circle(0.3ex);
		\filldraw (-1,-0.5)circle(0.3ex);
		\draw[very thick] (1.3,0)--(4.7,0);
		\draw(1.3,0)[]node[below=1pt]{$v$};	
		\draw(4.7,0)[]node[below=0pt]{$u$};	
		\draw(7.3,1.8)[]node[left=-2pt]{$u_1$};	
        \draw(7.3,0.5)[]node[left=-2pt]{$u_2$};	
        \draw(7.8,-2.2)[]node[left=1pt]{$u_{2k-2}$};	
        \draw(10.7,1.8)[]node[right=-2pt]{$w_1$};	
        \draw(10.7,0.5)[]node[right=-2pt]{$w_2$};	
        \draw(10.7,-2.2)[]node[right=1pt]{$w_{2k-2}$};	
        \draw(7.3,1.8)[]node[left=-2pt]{$u_1$};	 
    \draw[very thick] (6,-3.2)node[below=0pt]{$K_{2k-1}$};
    
   \draw[very thick] (0,-3.2)node[below=0pt]
   {$K_{\left \lceil \frac{n+1}2\right \rceil-k}$};
   
      \draw[very thick] (12,-3.2)node[below=0pt]
   {$K_{\left \lfloor \frac{n+1}2\right \rfloor-k}$};
    
	\end{tikzpicture}
\end{center}  

	\caption{$H$ is obtained 
		from vertex-disjoint graphs
	$K_{2k-1}$, 
	$K_{\left \lceil \frac{n+1}2\right \rceil-k}$
	and $K_{\left \lfloor \frac{n+1}2\right \rfloor-k}$
by adding edges $uv$ and $u_iw_i$ for $i=1,2,\cdots,2k-2$}
\label{f1-}
\end{figure}

We conclude this section by noting that 
Theorem~\ref{Thm1.6} fails if $\delta(G)\ge 2k$ is replaced 
by $\delta(G)\ge 2k-1$. Assume that $2\le k\le\frac{n}{6}$. Let $H$ be the graph obtained from vertex-disjoint graphs $K_{2k-1}$, $K_{\lceil\frac{n+1}{2}\rceil-k}$, and $K_{\left \lfloor \frac {n+1}2\right \rfloor-k}$ by adding an edge $uv$ with $u\in V(K_{2k-1})$ and $v\in V(K_{\lceil\frac{n+1}{2}\rceil-k})$, and by joining each vertex $u_i \in V(K_{2k-1})\setminus\{u\}$ to some vertex $w_i \in V(K_{\left \lfloor \frac {n+1}2\right \rfloor-k})$, as shown in Figure~\ref{f1-}, where the vertices $w_i$ are not necessarily distinct. 
It is easy to verify that $\delta(H)=2k-1$, $NC(H)\ge\frac{|H|-2}{2}$, and $u$ is a cut-vertex of $H$ with $d_{H}(u)=2k-1$. Lemma~\ref{le2.1} in the next section shows that $H$ has no $[2,k]$-STs for all $k\ge 2$.

\medskip
\noindent{\bf Remark:} 
It can be proved easily that $\sigma(G)\ge n-2$ implies that $NC(G)\ge\frac{n-2}{2}$. Moreover, one can verify that 
$n_0(k)\ge n_1(k)$, and 
thus $n\ge n_0(k)$
implies that 
$n\ge n_1(k)$. 
Hence, Theorem~\ref{Thm1.6} is a generalization of Theorem~\ref{Thm1.5}.


\section{Preliminaries}
We begin this section by showing that the graph $H$ illustrated in Figure~\ref{f1-} does not contain a $[2,k]$-ST for $k\ge2$.

\subsection{$H$ has no $[2,k]$-STs}

\begin{lemma} \label{le2.1}
For any $k\ge2$, the graph $H$ 
in Figure~\ref{f1-} does not contain a $[2,k]$-ST.	
\end{lemma}

\begin{proof}
	Suppose that $T$ is a $[2,k]$-ST of $H$. Since $u$ is a cut-vertex of $H$, by the definition of a $[2,k]$-ST, 
	we have $d_T(u)\ge k+1$. Then for each $u_i\in V(K_{2k-1})\setminus\{u\}$,
	we have 
	$$
	|N_T(u_i)\cap N_T[u]|\le1,
	$$
	and thus 
	\eqn{}
	{d_T(u_i)&\le&|(V(K_{2k-1})\cup\{w_i\})\setminus(N_T[u]\cup\{u_i\})|+|N_T(u_i)\cap N_T[u]|\nonumber\\ 
	&\le&|V(K_{2k-1})|+1
	-(|N_T[u]|+1)+1\nonumber\\ 
	&\le&2k-1+1-(k+3)+1=k-2<k+1,\nonumber
    }
	implying that all vertices in $V(K_{2k-1})\setminus\{u\}$ are leaves in $T$. 
	
	On the other hand, since $H-
	(V(K_{2k-1})\setminus\{u\})$ is disconnected, so is $T-(V(K_{2k-1})\setminus\{u\})$. Hence, there must exist a vertex in $V(K_{2k-1}) \setminus \{u\}$ that is not a leaf of $T$, a contradiction.\end{proof}

\vspace{3mm}

For the rest of this section, we assume that 
$G$ is a connected graph 
of order $n$ with $n>\delta(G)+1$ (i.e., $G$ is not complete)
and $NC(G)\geq\frac{n-2}{2}$
(i.e.,  the condition of (\ref{Thm1.6-e1}) holds). Denote by $[k]$ the set $\{1, 2, \ldots, k\}$. Let $u$ be a vertex in $G$ with 
$d(u)=\delta(G)$, $N(u)=\{u_i:i\in[\delta(G)]\}$ and $W=V(G)\backslash N[u]$. Since $n>\delta(G)+1$ and $d(u)=\delta(G)$, 
we have $W\neq\emptyset$ and $E(N(u),W)\neq\emptyset$.

\subsection{Non-complete graphs $G$ with $NC(G)\ge \frac{n-2}2$}

In this subsection, we will 
mainly provide a sufficient condition for a subset $S$
of $W$ such that $G[W\setminus S]$ has at most two components.
We will also provide some other 
conclusions. These conclusions will be used in the proof of Theorem~\ref{Thm1.6}.

\begin{lemma}\label{le2.2}
	If $\delta(G)<\frac{n-4k+4}{2}$, 
	then $\{u_i\in N(u): |N_W(u_i)|\leq k-1\}$  is a clique.
\end{lemma} 

\begin{proof} 
	Let $S=\{u_i\in N(u): |N_W(u_i)|\leq k-1\}$. 
	Suppose the $S$ is not a clique of $G$. Then 
	there exist two vertices 
	$u_p, u_q\in S$ such that $u_pu_q\notin E(G)$.
	Note that 
	$$
	N(u_p)\cup N(u_q)\subseteq(N[u]\backslash\{u_p,u_q\})\cup N_W(u_p)\cup N_W(u_q).
	$$ 
	Then 
	$$|N(u_p)\cup N(u_q)|\leq(\delta(G)+1-2)+2(k-1)
	<\frac{n-4k+4}{2}+2k-3
	=\frac{n-2}{2},$$
	which contradicts the assumption that $NC(G)\ge \frac{n-2}2$.
\end{proof}

\medskip

Note that for each $w\in W$, as $uw\notin E(G)$, $|N(u)\cup N(w)|\geq NC(G)\ge \frac{n-2}{2}$
by the given condition. 
Thus, 
\begin{equation}\label{e3} 
	\forall w\in W: \quad	d_{W}(w)\geq \frac{n-2}{2}-\delta(G).
\end{equation}

\begin{lemma} \label{le2.3}
	For any $S\subset W$ with $|S|<\frac{n+2}{4}-\delta(G)$, $G[W\backslash S]$ contains at most two components. 
\end{lemma} 

\begin{proof}
	Suppose that $G[W\backslash S]$ contains at least three components. Then there is a component $C_0$ of $G[W\backslash S]$ satisfying 
	\begin{equation} \label{e4} 
		|C_0|\leq\frac{|W\backslash S|}{3}
		=
		\frac{n-1-\delta(G)-|S|}{3}.
	\end{equation} 
	Note that for each $x\in C_0$, $N(x)\subseteq(C_0\backslash\{x\})\cup S\cup N(u)$, 
	implying that 
	\begin{eqnarray} \label{e5}
		|N(u)\cup N(x)|
		&\leq& |C_0|-1+|S|+\delta(G)
		\ \leq\ 
		\frac{n+2|S|+2\delta(G)-4}{3}
		\nonumber \\
		&<&
		\frac{n+2\cdot\frac{n+2}{4}-4}{3}
		\ =\ \frac{n-2}{2},
	\end{eqnarray} 
	a contradiction to the condition of (\ref{Thm1.6-e1}). So $G[W\backslash S]$ contains at most two components.
\end{proof}

\begin{lemma}\label{le2.4}
	Assume that $S\subset W$ and $G[W\backslash S]$ contains exactly two components. Then, for each component $C$, 
	\begin{enumerate}
		\item[(i)] $\frac{n}{2}-\delta(G)-|S|\leq|C|\leq\frac{n-2}{2}$, and
		
		\item[(ii)]  if $n\geq n_1(k)$
		and $|S|+\delta(G)< \frac{n-4k+6}4$,
		then $C-S'$ contains a
		$[2,k]$-ST
		for any $S'\subseteq C$ with $|S'|\leq2k-2$.
	\end{enumerate}
\end{lemma}

\begin{proof} (i) Let $C_1$ and $C_2$ be two components of $G[W\setminus S]$. For $i\in [2]$, let $x$ be any vertex in $C_i$. Then by (\ref{e3}), we have
	\equ{e6}
	{d_{C_i}(x)=d_{W\backslash S}(x)\geq d_W(x)-|S|\geq\frac{n-2}{2}
		-\delta(G)-|S|.}
	Then 
	$$
	|C_i|\geq d_{C_i}(x)+1\geq\frac{n}{2}-\delta(G)-|S|,
	$$ 
	implying that 
	$$
	|C_{3-i}|=|W\backslash S|-|C_{i}|\leq\left(n-1-\delta(G)-|S|\right)-\left(\frac{n}{2}-\delta(G)-|S|\right)=\frac{n-2}{2}.
	$$ 
	Hence (i) holds.

	(ii) We will apply Theorem~\ref{Thm1.4} to prove (ii). 
	Assume that $n\geq n_1(k)$
	and $|S|+\delta(G)< \frac{n-4k+6}4$.
	Let $S'\subseteq C$ 
	with $|S'|\le 2k-2$. 
	In order to prove this conclusion, by 
	Theorem~\ref{Thm1.4}, 
	it suffices to show that 
	$C-S'$ is connected 
	and $\delta(C-S')\ge c_k\sqrt{|C-S'|}$.
	
	\inclaim $C-S'$ is connected
	for any $S'\subseteq C$ 
	with $|S'|\le 2k-2$.
	
	Suppose that $C-S'$ is disconnected
	for some $S'\subseteq C$ 
	with $|S'|\le 2k-2$.
	Let $H_1$ and $H_2$ be any two components of $C-S'$. 
	For each $w'\in C-S'$, 
	by (\ref{e3}), we have
	{\begin{eqnarray}\label{e7}
		d_{H_j}(w')=d_W(w')-d_S(w')
		-d_{S'}(w')
		\ge	\frac{n-2}{2}
		-\delta(G)-|S|-|S'|.
	\end{eqnarray} }
	Assume that  $x\in H_1$ and $y\in H_2$.
	Then, by (\ref{e7}), 
	\equ{e8}
	{
		d_{H_1}(x)+d_{H_2}(y)
		\ge
		2\left(\frac{n-2}{2}
		-\delta(G)-|S|-|S'|\right)
		=n-2-2\left(\delta(G)+|S|\right)
		-2|S'|.
	}	
	Since $|S'|\le 2k-2$ and $\delta(G)+|S|< 
	\frac{n-4k+6}4$, (\ref{e8}) implies that 
	\begin{eqnarray}\label{e9}
		d_{H_1}(x)+d_{H_2}(y)
		&>&n-2-2\cdot \frac{n-4k+6}4-(2k-2)-|S'|
		=\frac{n-2}{2}-2-|S'|
		\nonumber \\
		&\ge& |C|-2-|S'|
		\	=\ |H_1|-1+|H_2|-1,
	\end{eqnarray} 
	a contradiction to the fact that 
	$d_{H_1}(x)\le |H_1|-1$ and 
	$d_{H_2}(y)\le |H_2|-1$, 
	where the last inequality 
	follows from the result in (i)
	that $|C|\le \frac{n-2}2$.
	Hence Claim 1 holds.
	

\inclaim $\delta(C-S')\ge c_k\sqrt{|C-S'|}$.

	
	Let $w\in C-S'$ with $d_{C-S'}(w)=\delta(C-S')$. Then 
	by (\ref{e6}), 
	\equ{e10}
	{
		\delta(C-S')=d_{C-S'}(w)
		=d_{C}(w)-d_{S'}(w)
		\ge \frac{n-2}{2}
		-\delta(G)-|S|-|S'|.
	}
	Since $\delta(G)+|S|<\frac{n-4k+6}4$, by (\ref{e10}), we have 
	\equ{e11}
	{\delta(C-S')
		>\frac{n-2}{2}-\frac{n-4k+6}{4}-|S'|
		=\frac{n+4k-10}4-|S'|.
	}
	Since $|S'|\le 2k-2$ and $k\ge2$, (\ref{e11}) implies that 
	\eqn{e15}
	{\delta(C-S')&>& \frac{n+4k-10}4-|S'|\ge \frac{n+4k-10}4-(2k-2)\nonumber \\
		&=&\frac{n-4k-2}4\ge \frac{n-12k+14}{4}
		\nonumber \\
		&\ge& c_k\sqrt{n}
		\ge c_k\sqrt{|C-S'|},
	}
	where the second last inequality follows from the condition $n\ge n_1(k)$, and the last inequality follows from 
	the fact that $n\ge |C-S'|$. Hence Claim 2 holds.
	
	By Claims 1 and 2, (ii) holds.
\end{proof}

\section{$i$-semi-$[2,k]$-T and $i$-quasi-$[2,k]$-T} 

A subtree $T$ of $G$ is called a $[2,k]$-T of $G$ if it has no vertices of degrees $2$ through $k$.





\begin{definition}\label{def2.6}
For any $i\ge 1$, 
an $i$-semi-$[2,k]$-T of $G$ is a $[2,k-1]$-T which has
exactly $i$  
vertices of degree $k$,
and an $i$-semi-$[2,k]$-ST of $G$ is an $i$-semi-$[2,k]$-T
which is a spanning tree of $G$. 
\end{definition}

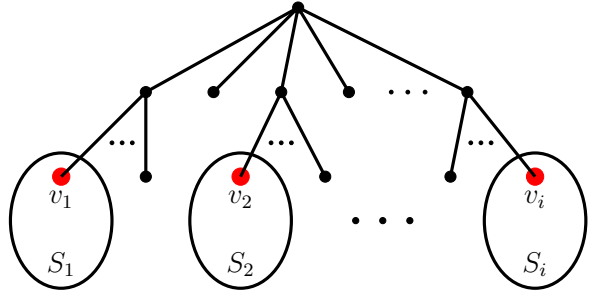
\begin{figure}[h]
	\begin{center}  
		\begin{tikzpicture}[scale=0.45]
			\tikzstyle{every node}=[font=\normalsize,scale=0.9]
			\draw[very thick](-1,4.2) ellipse [x radius=1.5cm, y radius=2cm];
			\draw[very thick](4.3,4.2) ellipse [x radius=1.5cm, y radius=2cm];
			\draw[very thick](13,4.2) ellipse [x radius=1.5cm, y radius=2cm];
			\filldraw (6,10.5)circle(0.9ex);
			\filldraw (1.5,8)circle(0.9ex);
			\filldraw (3.5,8)circle(0.9ex);
			\filldraw (5.5,8)circle(0.9ex);
			\filldraw (7.5,8)circle(0.9ex);
			\filldraw (8.75,8)circle(0.3ex);
			\filldraw (9.24,8)circle(0.3ex);
			\filldraw (9.75,8)circle(0.3ex);
			\filldraw (11,8)circle(0.9ex);
			\draw[very thick] (6,10.5)--(1.5,8);
			\draw[very thick] (6,10.5)--(3.5,8);
			\draw[very thick] (6,10.5)--(5.5,8);
			\draw[very thick] (6,10.5)--(7.5,8);
			\draw[very thick] (6,10.5)--(11,8);
			\filldraw[red] (-1,5.5)circle(1.4ex);
			\filldraw (1.5,5.5)circle(0.9ex);
			\filldraw[red] (4.3,5.5)circle(1.4ex);
			\filldraw (6.8,5.5)circle(0.9ex);
			\filldraw (10.5,5.5)circle(0.9ex);
			\filldraw[red] (13,5.5)circle(1.4ex);
			\filldraw (0.5,6.5)circle(0.3ex);
			\filldraw (0.8,6.5)circle(0.3ex);
			\filldraw (1.1,6.5)circle(0.3ex);
			\draw[very thick] (1.5,8)--(-1,5.5);
			\draw[very thick] (1.5,8)--(1.5,5.5);
			\filldraw (5.2,6.5)circle(0.3ex);
			\filldraw (5.5,6.5)circle(0.3ex);
			\filldraw (5.8,6.5)circle(0.3ex);
			\draw[very thick] (5.5,8)--(4.3,5.5);
			\draw[very thick] (5.5,8)--(6.8,5.5);
			\filldraw (11.7,6.5)circle(0.3ex);
			\filldraw (11.4,6.5)circle(0.3ex);
			\filldraw (11.1,6.5)circle(0.3ex);
			\draw[very thick] (11,8)--(10.5,5.5);
			\draw[very thick] (11,8)--(13,5.5);
			\filldraw (7.7,4.2)circle(0.5ex);
			\filldraw (8.5,4.2)circle(0.5ex);
			\filldraw (9.3,4.2)circle(0.5ex);
			\draw(-1,5.5)[]node[below=2pt]{$v_1$};	
			\draw(4.3,5.5)[]node[below=2pt]{$v_2$};	
			\draw(13,5.5)[]node[below=2pt]{$v_i$};	
			\draw(-1,3.7)[]node[below=2pt]{$S_1$};	
			\draw(4.3,3.7)[]node[below=2pt]{$S_2$};	
			\draw(13,3.7)[]node[below=2pt]{$S_i$};	
		\end{tikzpicture}
	\end{center}  
	\caption{An $i$-semi-$[2,k]$-T $T$ with $d_T(v_j)=k$ for all $j=1,2,\cdots,i$ }
	\label{f2-}
\end{figure}

By the definition of $i$-semi-$[2,k]$-T, the result below follows directly.

\begin{lemma}\label{le2.50}
	Let $i\ge 1$ and $T$ be an $i$-semi-$[2,k]$-T of $G$ 
	with $d_{T}(v_j)=k$
	for $j=1,2,\cdots,i$.
	Assume that 
	$S_1,\cdots,S_i$ are 
	disjoint subsets of $V(G)$ 
	with the properties that 
	for each $j=1,2,\cdots,i$,
	$G[S_j]$ has a $[2,k]$-ST
	and $S_j\cap V(T)=\{v_j\}$, as shown in Figure~\ref{f2-}.
	Then, $T$ can be extended 
	to a $[2,k]$-T $T'$ of $G$.
	 In particular, if 
	 $V(T)\cup \bigcup\limits_{1\le j\le i}S_j
	 =V(G)$, 
	 then $T'$ is a $[2,k]$-ST of $G$.
\end{lemma}

The next result holds obviously
and it will be repeatedly applied in this article. 
It will be first applied in Lemma~\ref{le2.8}. 

\begin{lemma}\label{le2.7}
	Let $G$ be a graph with $X\subset V(G)$ and $z\in V(G)\backslash X$. If $G[X]$ is connected and $1\leq |N_X(z)|<|X|$, 
	then there exists an induced path $zxy$ in $G[X\cup\{z\}]$, where $x,y\in X$.
\end{lemma}

In the following, assume that 
$G$ is a connected graph 
of order $n$ 
with 
$NC(G)\geq\frac{n-2}{2}$,
$u$ is a  vertex in $G$ with 
$d(u)=\delta(G)$
and $N(u)=\{u_i:i\in[\delta(G)]\}$ and $W=V(G)\backslash N[u]$.

\begin{lemma} \label{le2.8}
	Assume that $n\ge n_1(k)$
	and $G[W]$ is connected.
	For any  $S\subseteq W$ with $2\leq|S|<\frac{n-4k+10}{4}-\delta(G)$, 
	if $T$ is a $1$-semi-$[2,k]$-T
	with $V(T)=N[u]\cup S$
	and $v$ is a vertex in $S$ 
	with $d_T(v)=k$, 
	 as shown in Figure~\ref{f1}, 
	then 
	$G$ has a $[2,k]$-ST.
\end{lemma}

\begin{figure}[H]
	\begin{center}  
		\begin{tikzpicture}[scale=0.6]
			\tikzstyle{every node}=[font=\normalsize,scale=0.9]
			\filldraw (4.5,10)circle(0.7ex);
			\draw (7.1,9.8)[]node[above=2pt]{$u$~~~$(d_{T}(u)\ge k+1)$};			
			\draw(3,6.3)[]node[right=0pt]{$v$};	
			\draw(2.1,8.2)[]node[left=0pt]{$u'$};	
			\draw(1.1,8.2)[]node[left=0pt]{$(d_{T}(u')\ge k+1)$};	
			
			\filldraw (2,8)circle(0.7ex);
			\filldraw (3.4,8)circle(0.7ex);
			\filldraw (4.8,8)circle(0.7ex);
			\filldraw (7,8)circle(0.7ex);
			\filldraw (5.5,8)circle(0.35ex);
			\filldraw (5.9,8)circle(0.35ex);
			\filldraw (6.3,8)circle(0.35ex);%
			 			
			\filldraw(0.9,6.3)circle(0.7ex);
			\filldraw(1.5,6.3)circle(0.7ex);
			\filldraw(3,6.3)circle(0.7ex);
			
			\filldraw(2,6.3)circle(0.35ex);
			\filldraw(2.3,6.3)circle(0.35ex);
			\filldraw(2.6,6.3)circle(0.35ex);
			
			\filldraw(2,4.8)circle(0.7ex);
			\filldraw(2.6,4.8)circle(0.7ex);
			\filldraw(4,4.8)circle(0.7ex);
			
			\filldraw(3,4.8)circle(0.35ex);
			\filldraw(3.3,4.8)circle(0.35ex);
			\filldraw(3.6,4.8)circle(0.35ex);
			
			\node [] at (3,4.1) {$\underbrace{~~~~~~~~~~~}_{k-1}$};
			
			\filldraw(4.8,5.5)circle(0.7ex);
			\filldraw(5.7,5.5)circle(0.7ex);
			\filldraw(7.5,5.5)circle(0.7ex);
			\filldraw(7,5.5)circle(0.35ex);
			\filldraw(6.65,5.5)circle(0.35ex);
			\filldraw(6.3,5.5)circle(0.35ex);
			\draw[very thick] (4.5,10)--(2,8);		
			\draw[very thick] (4.5,10)--(3.4,8);		
			\draw[very thick] (4.5,10)--(4.8,8);		
			\draw[very thick] (4.5,10)--(7,8);	
			\draw[very thick] (2,8)--(0.9,6.3);
			\draw[very thick] (2,8)--(1.5,6.3);
			\draw[very thick] (2,8)--(3,6.3);
			\draw[very thick] (3,6.3)--(2,4.8);		
			\draw[very thick] (3,6.3)--(2.6,4.8);	
			\draw[very thick] (3,6.3)--(4,4.8);			
			\draw[thick, dashed](8.2,7)rectangle(0,3.3);	 				
			\draw[thick, dashed](4.4,6.7)rectangle(0.4,3.6);	
			\draw (1,5)[] node[right=0pt]{$S$};			
			\draw (8.5,5)[] node[right=0pt]{$W$};
		\end{tikzpicture}
	\end{center}	
	\caption{A 1-semi-$[2,k]$-T $T$ of $G$ with vertex set $N[u]\cup S$}
	\label{f1}
\end{figure}
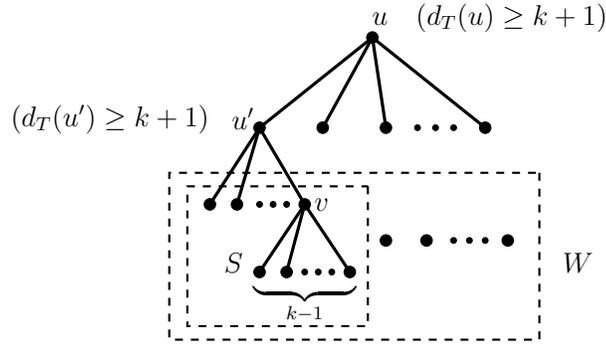 

\begin{proof} 
	Assume that $T$ is a $1$-semi-$[2,k]$-T of $G$ with $V(T)=N[u]\cup S$ and $d_{T}(v)=k$, where $v\in S$. 
	Let $S':=S\backslash\{v\}$. Then 
	$|S'|+\delta(G)
	<\frac{n-4k+6}{4}$.
	
	\incase $G[W\backslash S']$ is connected.
	
	By Lemma~\ref{le2.50},
	it suffices to show that 
	$G[W\setminus S']$
	contains a $[2,k]$-ST. By (\ref{e3}), we have 
	\eqn{e16}
	{
		\delta(G[W\backslash S'])
		&\geq&\frac{n-2}{2}-\delta(G)-|S'|>\frac{n-2}{2}-\frac{n-4k+6}{4}=\frac{n+4k-10}{4}\nonumber\\
		&>&\frac{n-12k+14}{4}\ge c_k\sqrt{n}
		\ge c_k\sqrt{|W\backslash S'|},
	}
	where the second last inequality follows from the condition that $n\ge n_1(k)$.
	Then, by Theorem~\ref{Thm1.4}, $G[W\backslash S']$ has a $[2,k]$-ST. Note that $V(T)\cap (W\backslash S')=\{v\}$ and $V(T)\cup (W\backslash S')=V(G)$. By Lemma~\ref{le2.50}, $G$ has a $[2,k]$-ST. 
	
	\vspace{3mm}
	\noindent {\bf Case 2}: $G[W\backslash S']$ is disconnected.
	\vspace{3mm}
	
	Since $|S'|+\delta(G)<\frac{n-4k+6}4$,
	by Lemma~\ref{le2.3}, 
	$G[W\backslash S']$ contains exactly two components,
	say $C_1$ and $C_2$, as shown in Figure~\ref{f2} (a). 
	Then, 
	\eqn{e17}
	{
	{k<\frac{n+4k-6}{4}}<\frac{n}{2}-\delta(G)-|S'|\leq|C_i|\le\frac{n-2}{2},
	}
where the last two inequalities 
are from Lemma~\ref{le2.4} (i).
	Note that $n\ge n_1(k)$, then Claim 1 below follows directly from Lemma~\ref{le2.4} (ii). 
	
	\setcounter{countclaim}{0}
	
	\inclaim 
	For any
	$i\in [2]$ and $S_0\subseteq C_i$
	with $|S_0|\le 2k-2$, $C_i-S_0$ contains a $[2,k]$-ST.  
	\vspace{3mm}
	
	Clearly, $v\in W\setminus S'=C_1\cup C_2$. 
	Assume that $v\in C_1$. 
	Since $G[W]$ is connected and $|S'|\geq1$, there exists some $x\in S'$ with $|N_{C_2}(x)|\geq1$ as shown in Figure~\ref{f2} (b) for $k=3$. 		
	
	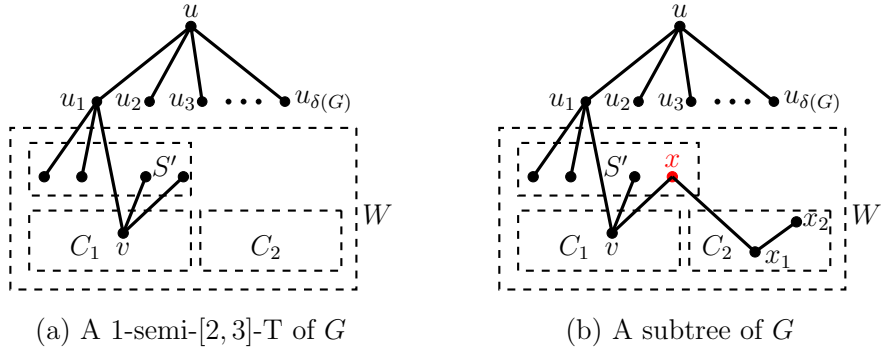
\begin{figure}[H]
		\begin{center}  
			\begin{tikzpicture}[scale=0.5]
				\tikzstyle{every node}=[font=\normalsize,scale=0.9]
				\filldraw (4.5,10)circle(0.75ex);
				\draw (4.5,10)[]node[above=0pt]{$u$};	
				\draw (2,8)[]node[left=0pt]{$u_1$};
				\draw (3.6,8)[]node[left=2pt]{$u_2$};
				\draw (4.9,8)[]node[left=0pt]{$u_3$};
				\draw (7,8)[]node[right=0pt]{$u_{\delta(G)}$};
				
				\draw (2.7,4.5)[]node[below=0pt]{$v$};

				\filldraw (2,8)circle(0.7ex);
				\filldraw (3.4,8)circle(0.7ex);
				\filldraw (4.8,8)circle(0.7ex);
				\filldraw (7,8)circle(0.7ex);
				
				\filldraw (5.5,8)circle(0.35ex);
				\filldraw (5.9,8)circle(0.35ex);
				\filldraw (6.3,8)circle(0.35ex);

				\filldraw(0.6,6)circle(0.75ex);
				\filldraw(1.6,6)circle(0.75ex);
				
				\filldraw(3.3,6)circle(0.75ex);
				\filldraw(4.3,6)circle(0.75ex);
				
				\filldraw(2.7,4.5)circle(0.75ex);

				\draw[very thick] (4.5,10)--(2,8);		
				\draw[very thick] (4.5,10)--(3.4,8);
				\draw[very thick] (4.5,10)--(4.8,8);
				\draw[very thick] (4.5,10)--(7,8);

				\draw[very thick] (2,8)--(0.6,6);			
				\draw[very thick] (2,8)--(1.6,6);		
				\draw[very thick]   (2,8)--(2.7,4.5);
				
				\draw[very thick]   (3.3,6)--(2.7,4.5);
				\draw[very thick]   (4.3,6)--(2.7,4.5);

				\draw[thick, dashed](8.9,7.3)rectangle(-0.3,3);	 				
				\draw[thick, dashed](4.5,5.1)rectangle(0.2,3.5);	   	
				\draw[thick, dashed](8.5,5.1)rectangle(4.75,3.5);		
				\draw[thick, dashed](4.5,5.5)rectangle(0.2,6.9);		
				\draw (1.7,4.7)[] node[below=0pt]{$C_1$};			
				\draw (6.5,4.7)[] node[below=0pt]{$C_2$};	
				\draw (3.2,6.3)[] node[right=0pt]{$S'$};			
				\draw (8.8,5)[] node[right=0pt]{$W$};
				\draw (4.5,2.5) node[below=0pt]{(a)~A 1-semi-$[2,3]$-T of $G$};	
				\filldraw (17.5,10)circle(0.75ex);
				\draw (17.5,10)[]node[above=0pt]{$u$};	
				\draw (15,8)[]node[left=0pt]{$u_1$};
				\draw (16.6,8)[]node[left=2pt]{$u_2$};
				\draw (17.9,8)[]node[left=0pt]{$u_3$};
				\draw (20,8)[]node[right=0pt]{$u_{\delta(G)}$};
				\draw (15.7,4.5)[]node[below=0pt]{$v$};
				\draw (19.5,3.8)[]node[right=0pt]{$x_1$};
				\draw (20.5,4.8)[]node[right=0pt]{$x_2$};
				
				\filldraw (15,8)circle(0.7ex);
				\filldraw (16.4,8)circle(0.7ex);
				\filldraw (17.8,8)circle(0.7ex);
				\filldraw (20,8)circle(0.7ex);
				\filldraw (18.5,8)circle(0.35ex);
				\filldraw (18.9,8)circle(0.35ex);
				\filldraw (19.3,8)circle(0.35ex);
				\filldraw(16.3,6)circle(0.75ex);
				\filldraw[red](17.3,6)circle(0.75ex);
				\filldraw(13.6,6)circle(0.75ex);
				\filldraw(14.6,6)circle(0.75ex);
				\filldraw(15.7,4.5)circle(0.75ex);
				\filldraw(19.5,4)circle(0.75ex);
				\filldraw(20.6,4.8)circle(0.75ex);
				
				\draw[very thick] (17.5,10)--(15,8);		
				\draw[very thick] (17.5,10)--(16.4,8);
				\draw[very thick] (17.5,10)--(17.8,8);
				\draw[very thick] (17.5,10)--(20,8);		 				
				\draw[very thick] (15,8)--(13.6,6);			
				\draw[very thick] (15,8)--(14.6,6);		
				\draw[very thick] (15,8)--(15.7,4.5);
				\draw[very thick] (16.3,6)--(15.7,4.5);
				\draw[very thick] (17.3,6)--(15.7,4.5);
				\draw[very thick] (17.3,6)--(19.5,4);
				\draw[very thick] (20.6,4.8)--(19.5,4);
				\draw[thick,dashed](21.9,7.3)rectangle(12.7,3);	 				
				\draw[thick,dashed](17.5,5.1)rectangle(13.2,3.5);	   	
				\draw[thick,dashed](21.5,5.1)rectangle(17.75,3.5);		
				\draw[thick,dashed](18,5.5)rectangle(13.2,6.9);		
				\draw (14.7,4.7)[] node[below=0pt]{$C_1$};			
				\draw (18.5,4.7)[] node[below=0pt]{$C_2$};	
				\draw (15.2,6.3)[] node[right=0pt]{$S'$};	
				\draw (21.8,5)[] node[right=0pt]{$W$};
				\draw (17.3,6)[red] node[above=0pt]{$x$};
				\draw (17.5,2.5) node[below=0pt]{(b)~A subtree of $G$};	
			\end{tikzpicture}
		\end{center}	
		\caption{Two subtrees of $G$ when $k=3$}
		\label{f2}
	\end{figure} 
	
	\vspace{2mm}
	\noindent {\bf Subcase 2.1}: $|N_{C_2}(x)|=|C_2|$.
	\vspace{2mm}
	
	By (\ref{e17}), we have $|N_{C_2}(x)|=|C_2|>k$. Then
	$T$ can be extended to a $1$-semi-$[2,k]$-T $T'$ with edge set $E(T)\cup E(x,C_2)$ and $d_{T'}(v)=k$. 
	Note that $C_1\cap V(T')=\{v\}$ and $C_1\cup V(T')=V(G)$.
	By Claim 1, $C_1$ has a $[2,k]$-ST $T_1$. Then, by Lemma~\ref{le2.50}, $E(T')\cup E(T_1)$ induces 
	a $[2,k]$-ST of $G$.
	
	\vspace{3mm}
	\noindent {\bf Subcase 2.2}: $1\le|N_{C_2}(x)|<|C_2|$.
	\vspace{3mm}
	
	By Lemma~\ref{le2.7}, there exist $x_1,x_2\in C_2$ such that $xx_1,x_1x_2\in E(G)$ and $xx_2\notin E(G)$, as shown in Figure~\ref{f2} (b).

\inclaim $d_{C_2}(x_1)>2k$
and $d_{C_1}(x)+d_{C_2}(x)>k$.

	
	By $(\ref{e3})$, we have
	\eqn{e18}
		{
		d_{C_2}(x_1)&=&d_W(x_1)-d_{S'}(x_1)\ge\frac{n-2}{2}-\delta(G)-|S'|\nonumber\\
		&>&\frac{n-2}{2}-\frac{n-4k+6}{4}=\frac{n+4k-10}4>2k
		}
		and 
		\eqn{e19}
		{d_{C_1}(x)+d_{C_2}(x)&=&d_W(x)-d_{S'}(x)
			\ge 
			\left(\frac{n-2}{2}
			-\delta(G)\right)
			-\left(|S'|-1
			\right)
			\nonumber \\
			&>&
			\frac{n-2}2-\delta(G)-|S'|+1
			>\frac{n+4k-6}{4}>k.}
	Thus, Claim 2 holds.
		
	\vspace{2mm}	
	Assume that $|N_{C_2}(x)|=a$. Let $S_1\subseteq N_{C_2}(x)\setminus\{x_1\}$ with $|S_1|=\min\{a-1,k-1\}$. Then $|S_1|\le k-1$.
	By (\ref{e18}), $d_{C_2}(x_1)>2k$, and so 
	$|N_{C_2}(x_1)\setminus(S_1\cup\{x_2\})|\ge k-2$. 
	Let $S_2\subseteq N_{C_2}(x_1)\setminus(S_1\cup\{x_2\})$ with $|S_2|=k-2$. Then 
	$$
	|S_1\cup S_2\cup\{x_2\}|\le k-1+k-2+1=2k-2.
	$$

	By (\ref{e19}), $d_{C_1}(x)>k-a$, and so there are at least $k-a$ vertices in $N_{C_1-v}(x)$. Let $S_3\subseteq N_{C_1-v}(x)$ with $|S_3|
	=k-1-|S_1|\le k-a$. 
Now we extend $T$ to a 
	2-semi-$[2,k]$-T $T''$ with 
	vertex set 
	$$
	V(T)\cup S_1\cup S_2\cup S_3\cup\{x_1,x_2\}
	$$ 
	and edge set 
	$$
	E(T)\cup E(x,S_1\cup S_3\cup\{x_1\})\cup E(x_1,S_2\cup\{x_2\}),
	$$
	and $v$ and $x_1$ 
	are the only vertices 
	in $T''$ of degree $k$ in $T''$.

	By Claim 1, $C_1-S_3$ has a $[2,k]$-ST, and $C_2-S_1\cup S_2\cup\{x_2\}$ has a $[2,k]$-ST. Since $V(T'')\cap (C_1-S_3)=\{v\}$, $V(T'')\cap (C_2-S_1\cup S_2\cup\{x_2\})=\{x_1\}$ and $V(T'')\cup(C_1-S_3)\cup(C_2-S_1\cup S_2\cup\{x_2\})=V(G)$. Then by Lemma~\ref{le2.50}, $G$ has a $[2,k]$-ST.
 \end{proof} 

\begin{definition}
A subtree $T$ of $G$ is called an \textit{i}-quasi-$[2,k]$-T if it contains exactly $i$ vertices whose degrees are between $2$ and $k$. Specifically, for a 1-quasi-$[2,k]$-T, denoted by $T_w$, the vertex $w$ is the only one satisfying this degree condition.
\end{definition}

\begin{lemma}\label{le2.9}
	Assume that $n\geq n_1(k)$,  $\delta(G)<\frac{n-12k+14}{4}$,  and 
	$G[W]$ contains exactly two components.
	For any component $C$ 
	of $G[W]$
	and $v\in N(u)$, 
	\begin{enumerate}
		\item[(i)] if $d_{C}(v)\in\{1,|C|\}$, then $C+v$ has a $[2,k]$-ST $T$ with $d_{T}(v)=d_{C}(v)$, and
		
		\item[(ii)] if $2\le d_{C}(v)<|C|$,
	 then $C+v$ has a $1$-quasi-$[2,k]$-ST 
	 $T_{v}$ 
	 with $d_{T_v}(v)=\min\{k,d_{C}(v)\}$.
	\end{enumerate} 
\end{lemma}

\begin{proof}
	Since $\delta(G)<\frac{n-12k+14}4$,
	by Lemma~\ref{le2.4} (i), we have 
	\eqn{e17*}{
	|C|\ge\frac{n}2-\delta(G)>\frac{n+12k-14}4>3k.
    }
    
	If $d_{C}(v)=|C|$, then $C+v$ has a $[2,k]$-ST $T$ with edge set $E(v,C)$ and $d_{T}(v)=d_{C}(v)>3k$. So, in the following, we assume that $1\le d_{C}(v)<|C|$.
	Then by Lemma~\ref{le2.7}, there exist $x_1,x_2\in C$ such that $vx_1,x_1x_2\in E(G)$ and $vx_2\notin E(G)$.
	
	
	\setcounter{countclaim}{0}
	
	Now we are going to apply Lemma~\ref{le2.4} (ii) to 
	prove Claim 1 below.
	
	\inclaim There exist $S_1\subseteq N_C(v)\setminus \{x_1\}$ 
	with $|S_1|=\min\{k-1,d_C(v)-1\}$
	and 
	$S_2 \subseteq N_{C}(x_1)\setminus(S_1\cup\{x_2\})$ with $|S_2|=k-2$
	such that 
	$C-(S_1\cup S_2\cup\{x_2\})$ has a $[2,k]$-ST $T'$.

	\vspace{3mm}
	
	For each $x\in C$, by $(\ref{e3})$,
	\equ{e20}
	{
		d_{C}(x)\ge\frac{n-2}{2}-\delta(G)>\frac{n+12k-18}{4}>3k.
	}
	
	Since $1\le d_{C}(v)<|C|$, $|N_{C}(v)\setminus\{x_1\}|\ge\min\{k-1,d_{C}(v)-1\}$. Let $S_1 \subseteq N_{C}(v)\setminus\{x_1\}$ with $|S_1|=\min\{k-1,d_{C}(v)-1\}$. Then $|S_1|\le k-1$. Clearly, if $d_{C}(v)=1$, then $S_1=\emptyset$. By (\ref{e20}), we have $d_{C}(x_1)>3k$, and so $|N_{C}(x_1)\setminus(S_1\cup\{x_2\})|\ge k-2$. Let $S_2 \subseteq N_{C}(x_1)\setminus(S_1\cup\{x_2\})$ with $|S_2|=k-2$. Then 
	$$
	|S_1\cup S_2\cup\{x_2\}|\le k-1+k-2+1=2k-2.
	$$
	By Lemma~\ref{le2.4} (ii),
	$C-S_1\cup S_2\cup\{x_2\}$ has a $[2,k]$-ST $T'$,
	and thus Claim 1 holds.
	
	If $d_{C}(v)=1$, then $C+v$ has a $[2,k]$-ST $T''$ with edge set $E(T')\cup E(x_1,S_2\cup\{v,x_2\})$ and $d_{T''}(v)=d_{C}(v)$; and if $2\le d_{C}(v)<|C|$, then  $C+v$ has a 1-quasi-$[2,k]$-ST $T_{v}$ with edge set $E(T')\cup E(v,S_1\cup\{x_1\})\cup E(x_1,S_2\cup\{x_2\})$ and $d_{T_{v}}(v)=|S_1|+1=\min\{k-1,d_{C}(v)-1\}+1=\min\{k,d_{C}(v)\}$. 
	\end{proof}

\section{Proof of Theorem~\ref{Thm1.6}}

Let $k\ge2$ be an integer, and let $G$ be a connected graph 
of order $n\ge n_1(k)$.
By Theorem~\ref{Thm1.4}, if $\delta(G)\ge c_k\sqrt{n}$, 
then $G$ has a $[2,k]$-ST. Thus, in order to complete the proof of Theorem~\ref{Thm1.6}, it suffices to verify the following statement.

\begin{prop}\label{prop3.1}
	Let $G$ be a connected graph of order $n\ge n_1(k)$. If $2k\le\delta(G)<c_k\sqrt{n}$
	and $NC(G)\ge \frac{n-2}2$, then $G$ contains a $[2,k]$-ST. 
\end{prop} 

From now on, let $G$ be a graph satisfying the conditions of Proposition~\ref{prop3.1}. Denote $\delta:= \delta(G)$. Let $u\in V(G)$ with $d(u)=\delta$, $N(u)=\{u_i:i\in[\delta]\}$ and $W=V(G)\setminus N[u]$. Since $n\ge n_1(k)$, 
$$|W|=n-1-\delta>n-1-c_k\sqrt{n}\ge n-1-\frac{n-12k+14}{4}=\frac{3n+12k-18}{4}>3k.
$$
implying that $W\neq\emptyset$.
Since $G$ is connected, we have  $E(N(u),W)\neq\emptyset$.

The rest of the proof
is divided into two cases,
which will be 
presented in two subsections.

\subsection{$G[W]$ is connected}
Let $U_i=N_W(u_i)$ for each $i\in [\delta]$. Assume $|U_1|=\max\limits_{i\in[\delta]}|U_i|$. Then $|U_1|\ge1$.

\setcounter{countcase}{0}

\incase $|U_1|=|W|$.

In this case, $G$ has a spanning tree 
$T$ with edge set $E(u,N(u))\cup E(u_1,W)$.
Clearly, 
$u$ and $u_1$ are the only vertices in $T$ 
of degrees larger than $1$, and $d_T(u)=\delta\ge 2k$ 
and $d_T(u_1)=|W|+1>3k$. 
Thus,  $T$ is a $[2,k]$-ST of $G$.

\incase $1\le|U_1|<|W|$.

Since $|U_1|<|W|$, by Lemma~\ref{le2.7}, there exist $x_1,x_2\in W$ such that $u_1x_1\in E(G)$, $x_1x_2\in E(G)$ and $u_1x_2\notin E(G)$. By $(\ref{e3})$, 
\equ{e21}
{d_W(x_1)\ge\frac{n-2}{2}-\delta
	>\frac{n-2}{2}-c_k\sqrt{n}\ge\frac{n-2}{2}-\frac{n-12k+14}{4}=\frac{n+12k-18}{4}>3k.
}		

We consider the following two subcases.

\vspace{2mm}
\noindent{\bf Subcase 2.1}: $1\le|U_1|\le k-1$.
\vspace{2mm}

By (\ref{e21}), we have $d_W(x_1)>3k$, and so there are at least $k-2$ vertices in $N_W(x_1)\setminus\{x_2\}$. Let $S_1\subseteq N_W(x_1)\setminus\{x_2\}$ with $|S_1|=k-2$. On the other hand, by the assumption, we have $|U_i|\le k-1$ for each $i\in[\delta]$. Then by Lemma~\ref{le2.2}, $N(u)$ is a clique in $G$. Thus $G[N[u]\cup S_1\cup\{x_1,x_2\}]$ has a $1$-semi-$[2,k]$-ST $T$ with edge set 
$$
E\left(u_1,(N[u]\setminus\{u_1\}\right)\cup\{x_1\})\cup E(x_1,S_1\cup\{x_2\}).
$$
Clearly, all vertices in 
$V(T)\setminus \{u_1,x_1\}$ 
are leaves  in $T$ 
and 
 $$
 d_T(u_1)=|N[u]|-1+1
 =\delta+1\ge2k+1
 $$
  and $$
  d_T(x_1)
  =|S_1|+2=k-2+2=k.
  $$
Observe that 
\eqn{}
{
\frac{n-4k+10}{4}-\delta
&>&\frac{n-4k+10}{4}-c_k\sqrt{n}\ge\frac{n-4k+10}{4}-\frac{n-12k+14}{4}\nonumber\\
&=&2k-1>k=|S_1\cup\{x_1,x_2\}|\ge2.\nonumber
}
Since $G[W]$ is connected, by Lemma~\ref{le2.8}, $G$ has a $[2,k]$-ST.

\vspace{2mm}

\noindent{\bf Subcase 2.2}: $k\le|U_1|<|W|$.

\vspace{2mm}

Let $S_2\subseteq U_1\setminus\{x_1\}$ with $|S_2|=k-1$. By $(\ref{e21})$, we have $d_W(x_1)>3k$, and so there are at least $k-2$ vertices in $N_{W}(x_1)\setminus(S_2\cup\{x_2\})$. Let $S_3\subseteq N_{W}(x_1)\setminus(S_2\cup\{x_2\})$ with $|S_3|=k-2$. Thus $G[N[u]\cup S_2\cup S_3\cup\{x_1,x_2\}]$ has a 1-semi-$[2,k]$-ST $T'$ with edge set
$$E(u,N(u))\cup E(u_1,S_2\cup\{x_1\})\cup E(x_1,S_3\cup\{x_2\}).
$$
Clearly, 
all vertices in $V(T')\setminus 
\{u,u_1,x_1\}$ are leaves 
in $T'$, and 
$$
d_{T'}(u)=|N(u)|
=\delta\ge2k, \quad 
d_{T'}(u_1)=|S_2|+2=k+1,
\quad   d_{T'}(x_1)=|S_3|+2=k.
$$
Observe that 
\eqn{}
{
\frac{n-4k+10}{4}-\delta
&>&\frac{n-4k+10}{4}-c_k\sqrt{n}\ge\frac{n-4k+10}{4}-\frac{n-12k+14}{4}\nonumber\\
&=&2k-1=|S_2\cup S_3\cup\{x_1,x_2\}|\ge2.\nonumber
}
Since $G[W]$ is connected, by Lemma~\ref{le2.8}, $G$ has a $[2,k]$-ST.

Hence Proposition~\ref{prop3.1} holds when $G[W]$ is connected.

\subsection{$G[W]$ is disconnected}

By Lemma~\ref{le2.3}, $G[W]$ contains exactly two components $C_1$ and $C_2$. For $i=1,2$, 
let $N^i(u):=\{u_j:1\le j\le \delta, E(u_j, C_i)\ne\emptyset\}$. Since $G$ is connected, $N^i(u)\neq\emptyset$ for both $i=1,2$. In the following, we prove Proposition~\ref{prop3.1} for the two subcases: $N^1(u)\cap N^2(u)\ne \emptyset$
and $N^1(u)\cap N^2(u)=\emptyset$, respectively. 		 

\subsubsection{$N^1(u)\cap N^2(u)\neq\emptyset$}

In this subsection, we assume that $u_1\in N^1(u)\cap N^2(u)$ with $d(u_1)\ge d(u_j)$ for each $u_j\in N^1(u)\cap N^2(u)$. Clearly,  $d_{C_1}(u_1)\ge1$ and $d_{C_2}(u_1)\ge1$. By Lemma~\ref{le2.9}, for each $i\in[2]$, if $d_{C_i}(u_1)\in\{1,|C_i|\}$, then $C_i+u_1$ has a $[2,k]$-ST $T_i$ with $d_{T_i}(u_1)=d_{C_i}(u_1)$; and if $2\le d_{C_i}(u_1)<|C_i|$, then $C_i+u_1$ has a $1$-quasi-$[2,k]$-ST $T^i_{u_1}$ with $d_{T^i_{u_1}}(u_1)=\min\{k,d_{C_i}(u_1)\}$.

\setcounter{countcase}{0}

\incase $d_{C_1\cup C_2}(u_1)\ge k$.

In this case, $G$ has a $[2,k]$-ST $T$ with edge set $E(T_1')\cup E(T_2')\cup E(u,N(u))$, where
$T_i'=T_i$ if $d_{C_i}(u_1)\in \{1,|C_i|\}$,
and $T_i'=T^i_{u_1}$ 
otherwise, for $i=1,2$. Clearly, $d_T(u)=|N(u)|=\delta\ge2k$ and $d_T(u_1)=d_{T_1'}(u_1)+d_{T_2'}(u_1)+1\ge k+1$. 

\incase $d_{C_1\cup C_2}(u_1)\le k-1$. 

In this case, we have
$$
d_{N(u)}(u_1)\ge\delta-d_{C_1\cup C_2}(u_1)-1\ge 2k-(k-1)-1=k.
$$ 
Let $S'\subseteq N_{N(u)}(u_1)$ with $|S'|=k-2$. Note that $d_{C_1\cup C_2}(u_1)\ge2$. Thus, $G$ has a $[2,k]$-ST $T'$ with edge set $E(u,N(u)\setminus S')\cup E(u_1,S')\cup E(T_1')\cup E(T_2')$, where
$T_i'=T_i$ if $d_{C_i}(u_1)\in \{1,|C_i|\}$,
and $T_i'=T^i_{u_1}$ 
otherwise, for $i=1,2$. 
Clearly, 
$$
d_{T'}(u)=|N(u)\setminus S'|\ge2k-(k-2)=k+2,
\quad  d_{T'}(u_1)=d_{T_1'}(u_1)+d_{T_2'}(u_1)+|S'|+1\ge k+1.
$$

Hence Proposition~\ref{prop3.1} holds when $N^1(u)\cap N^2(u)\ne \emptyset$.

\subsubsection{$N^1(u)\cap N^2(u)=\emptyset$}

Recall that $N^1(u)\neq\emptyset$ and $N^2(u)\neq\emptyset$. Assume that 
for $i=1,2$, $u_i\in N^i(u)$ with $d_{C_i}(u_i)\ge 
d_{C_i}(u_j)$ for each $u_j\in N^i(u)$. Without loss of generality, assume that $d_{C_1}(u_1)\ge d_{C_2}(u_2)$. 

By Lemma~\ref{le2.9}, for each $i\in[2]$, if $d_{C_i}(u_i)\in\{1,|C_i|\}$, then $C_i+u_i$ has a $[2,k]$-ST $T_i$ with $d_{T_i}(u_i)=d_{C_i}(u_i)$; and if $2\le d_{C_i}(u_i)<|C_i|$, then $C_i+u_i$ has a $1$-quasi-$[2,k]$-ST $T^i_{u_i}$ with $d_{T^i_{u_i}}(u_i)=\min\{k,d_{C_i}(u_i)\}$.

\setcounter{countcase}{0}

\incase $d_{C_2}(u_2)\ge k$.

In this case, 
$d_{C_1}(u_1)\ge d_{C_2}(u_2)
\ge k$.
 Then $G$ has a $[2,k]$-ST $T$ with edge set $E(T_1')\cup E(T_2')\cup N(u,N(u))$, where
 $T_i'=T_i$ if $d_{C_i}(u_i)\in \{1,|C_i|\}$,
 and $T_i'=T^i_{u_i}$ 
 otherwise, for $i=1,2$.
 
 \incase  $d_{C_1}(u_1)\le k-1$.
 
 In this case, 
 by the choice of vertices $u_1$ and $u_2$, $d_{C_1\cup C_2}(u_i)\le k-1$ for each $i\in[\delta]\setminus\{1,2\}$, and thus $N(u)$ is a clique by Lemma~\ref{le2.2}.
 It  follows that $G$ has a $[2,k]$-ST $T'$ with edge set 
$$
\bigcup_{j=2}^k\{u_1u_j\}\cup\bigcup_{i=k+1}^{\delta}\{u_2u_i\}\cup E(T_1')\cup E(T_2')\cup\{uu_1\},
$$
 where
 $T_i'=T_i$ if $d_{C_i}(u_i)\in \{1,|C_i|\}$,
 and $T_i'=T^i_{u_i}$ 
 otherwise, for $i=1,2$. Clearly, $d_{T'}(u_1)\ge k+1$ and $d_{T'}(u_2)\ge k+1$. 
 
 \incase 
 $d_{C_1}(u_1)\ge k$ and $d_{C_2}(u_2)\le k-1$. 

Since $d(u_2)\ge\delta\ge2k$ and $d_{C_2}(u_2)\le k-1$, we have 
$$
d_{N(u)}(u_2)\ge \delta-d_{C_2}(u_2)-1\ge2k-(k-1)-1=k.
$$ 
So there are at least $k-1$ vertices in $N_{N(u)}(u_2)$. Let $S'\subseteq N_{N(u)}(u_2)$ with $|S'|=k-1$. Then $G$ has a $[2,k]$-ST $T''$ with edge set 
$E(T_1')\cup E(T_2')\cup E(u,N(u)\setminus S')\cup E(u_2,S')$, where
$T_i'=T_i$ if $d_{C_i}(u_i)\in \{1,|C_i|\}$,
and $T_i'=T^i_{u_i}$ 
otherwise, for $i=1,2$. Clearly, $d_{T''}(u_2)=|S'|+1+d_{T_2'}(u_2)\ge k+1$, and $d_{T''}(u)=|N(u)\setminus S'|\ge2k-(k-1)=k+1$.

Hence Proposition~\ref{prop3.1} holds when $N^1(u)\cap N^2(u)=\emptyset$. 

This completes the proof of Theorem~\ref{Thm1.6}.

\section*{Declaration of Competing Interest}
The authors declare that they have no known competing financial interests or personal relationships that could have appeared to influence the work reported in this paper.

\end{document}